\documentclass[a4paper, 11pt]{amsart}
\usepackage{amssymb}
\usepackage{amsmath}
\usepackage{mathrsfs}
\usepackage{bbm}
\usepackage[dvips]{graphicx}
\usepackage{hyperref}
\usepackage{color}
\usepackage{multirow}
\usepackage{tikz}
\usetikzlibrary{arrows}
\usepackage{soul}
\usepackage[normalem]{ulem}
\usepackage{stackengine}

\numberwithin{equation}{section}

\newcommand{\one}{\mathbbm{1}}

\newcommand{\be}{\begin{equation}}
\newcommand{\ee}{\end{equation}}
\newcommand{\ds}{\displaystyle}

\newcommand{\iden}{\mathbbm{1}}

\newcommand{\Z}{\mathbb{Z}}

\newcommand{\X}[1]{(X_{#1})_{#1 \in \mathbb{Z}_{\geq 0}}}
\newcommand{\Xu}[1]{(X^{\text{(u)}}_{#1})_{#1 \in \mathbb{Z}_{\geq 0}}}

\newcommand{\ann}{\mathrm{LAnn}}

\newcommand{\ind}{\mathrm{Ind}}
\newcommand{\GL}{\mathrm{GL}}

\newcommand{\matt}[4]{\begin{pmatrix}
#1 & #2 \\ #3 & #4 \end{pmatrix}}

\newcommand{\Irr}{\mathrm{Irr}}
\newcommand{\Stab}{\mathrm{LStab}}
\newcommand{\mring}[1]{M_2(\mathbb F_#1)}

\newcommand\tss[2][.5ex]{%
  \def\stacktype{S}%
  \belowbaseline[#1]{\scriptsize#2}%
}
\newcommand{\rhoq}{\rho\tss[-0.2ex]{$\one_q$}}

\newtheorem{thm}{Theorem}[section]
\newtheorem{prop}[thm]{Proposition}
\newtheorem{cor}[thm]{Corollary}
\newtheorem{Lemma}[thm]{Lemma}
\newtheorem{defn}[thm]{Definition}
\newtheorem{Remark}[thm]{Remark}

\newtheorem{example}[thm]{Example}

\title[Random motion on finite noncommutative rings]{Random motion on finite rings, II: \\ Noncommutative rings}
\author{Arvind Ayyer}
\address{AA: Department of Mathematics, 
Indian Institute of Science, Bangalore  560012, India.}
\email{arvind@iisc.ac.in}

\author{Pooja Singla}
\address{PS: Department of Mathematics, 
Indian Institute of Science, Bangalore  560012, India.}
\email{pooja@iisc.ac.in}

\date{\today}
\begin{document}

\begin{abstract}
We extend our previous study of Markov chains on finite commutative rings 
(\href{https://arxiv.org/abs/1807.04082}{\texttt{arXiv:1605.05089}}) to arbitrary finite rings with identity. At each step, we either add or multiply by a randomly chosen element of the ring, where the addition (resp. multiplication) distribution is uniform (resp. conjugacy invariant).
We prove explicit formulas for some of the eigenvalues of the transition matrix and give lower bounds on their multiplicities. 
We also give recursive formulas for the stationary distribution and prove that the mixing time is bounded by an absolute constant. 
For the matrix rings $\mring q,$ we compute the entire spectrum explicitly using the representation theory of $\GL_2(\mathbb F_q),$ as well as the stationary probabilities. 

\end{abstract}


\maketitle

\section{Introduction}

Random walks on finite noncommutative groups have been a subject of extensive study. The most famous examples are perhaps those concerning random walks on the symmetric group $S_n$ under the label of {\em card-shuffling problems}. Starting with the random-transpositions model~\cite{diaconis-shahshahani-1981}, a lot of progress has been made in understanding the mixing times for such random walks. In particular, most of these shuffling algorithms satisfy the so-called cutoff phenomenon. See~\cite{diaconis-1998} for a survey of card-shuffling random walks.

In a completely different direction, the problem of efficient generation of quasi-random integers led to the development of random walks with fast mixing properties on $\Z_n$~\cite{chung-diaconis-graham-1987,hildebrand-1993}. For these walks, the fact that $\Z_n$ has the structure of a ring is crucial; both  additive and multiplicative operations were used to define the random walks. These ideas were extended to random walks on the vector spaces $\Z_p^d$~\cite{asci-2009a,asci-2009b,hildebrand-mccollum-2008}. In an earlier work, we have studied a general class of random walks on finite commutative rings~\cite{AS-2016}. This has since been extended to random walks on modules of finite commutative rings~\cite{ayyer-steinberg-2017}.

In this work, we combine these two threads of ideas and consider random walks on finite rings generalising our work on finite commutative rings.
The setup is as follows. Let $R$ be a finite ring with identity. We define two probability distributions on $R.$ The first, $U$ is the uniform distribution. For the second, we need some notation. We say that two elements $a,b \in R$ belong to the same {\em similarity class} if there exists a left-invertible element $u \in R$ such that $b = u \, a \, u^{-1}.$ One also says that in such a case $a$ and $b$
are {\em similar} or {\em conjugate}. We remark that as $R$ is a finite ring, an element is left invertible if and only if it is invertible on both sides. Let $Q$ be any distribution on $R$ which is constant on similarity classes.
Then define a discrete-time Markov chain $\X t$ on $R$ as follows. At each step, toss an independent coin with Heads probability $\alpha \in (0,1).$ If the coin lands Heads, set $X_{t+1} = X_t + Y$ where $Y$ is chosen independently according to $U,$ and otherwise, set $X_{t+1} = X_t \cdot Z,$ where $Z$ is chosen independently according to $Q.$ Note that if $R$ is commutative, all similarity classes are singletons and there is no restriction on the distribution $Q.$ 

Let the transition matrix of $\X t$ on $R$ be denoted $M_R.$ In other words, $M_R$ is the matrix indexed by the elements of $R$ with $M_R(a,b)$ is the one-step transition probability of getting from $a$ to $b.$
Let $\iden_{|R|}$ be the column vector of length $|R|$ with all ones.
Then, we can write
\begin{equation}
\label{transition matrix}
M_R = \frac{\alpha }{|R|} \iden_{|R|} \iden_{|R|}^{\text{tr}} + (1-\alpha) B_R,
\end{equation}
where $B_R$ encodes only multiplicative transitions given by
\begin{equation}
\label{mult part of transition matrix}
B_R(a,b) = \sum_{\substack{x \in R \\ xa = b}} Q(x).
\end{equation}

It then follows that if the eigenvalues of $B_R$ are given by 
$\lambda_1 = 1, \lambda_2,\dots,\lambda_{|R|},$ the eigenvalues of $M_R$ are given by $\lambda_1 = 1, (1 - \alpha) \lambda_2,\dots,(1 - \alpha) \lambda_{|R|}$; see, for example, \cite[Corollary 3.1]{ding-zhou-2007}.

The plan of the rest of the paper is as follows. We prove a formula for some of the eigenvalues and their multiplicities of the transition matrix $M_R$ of the chain $\X t$ in Section~\ref{sec:spectrum}. 
We will show that our results allow us to compute the entire spectrum of $M_R$ for $R = \mring q$ and certain subrings.
We note that we are not able to describe all eigenvalues in general using our techniques. 
However, we have considered many classes of finite rings and in all small examples, we are able to determine all eigenvalues using ad-hoc methods.

Since all entries of $M_R$ are positive, it immediately follows that $\X t$ is irreducible and aperiodic, and hence has a unique stationary distribution.
We will give a recursive formula for the stationary distribution in Section~\ref{sec:statdist}. We will then show that the chain mixes in constant time in Section~\ref{sec:mixing}.
Lastly, we will consider the example of the matrix ring $\mring q$ in detail in Section~\ref{sec:gl2}. In particular, we will consider the ``uniform multiplication'' chain $\Xu t,$ which is similar in definition to $\X t,$ except that both $Y$ and $Z$ are chosen according to $U$ at each step. We will then compute the exact stationary distribution for $\Xu t$ in Section~\ref{sec:m2 statdist}.

\section{Spectrum}
\label{sec:spectrum}
Let $R$ be a finite ring with identity. The set of invertible elements of $R$ forms a group, which we denote by $U_R.$
For $a \in R,$ let $I_a$ denote the principal left ideal generated by $a.$
Let $\phi$ be a fixed set of generators of distinct principal ideals of $R.$ Let $S_a$ be the set of all elements of $I_a$ that generate $I_a$ as a left ideal.
For $a \in R,$ let $\ann(a) = \{ x \in R \mid xa = 0 \}$ be the left annihilator of $a$ and 
$\Stab(a) = \{ x \in U_R \mid xa = a \}$ be the set
 of elements of $U_R$ that fix $a$ by acting from the left. Then $\Stab(a)$ is a subgroup of $U_R$ for all $a \in R.$ The set of equivalence classes of finite dimensional complex irreducible representations of $U_R$ is denoted by $\Irr(U_R).$ The one-dimensional trivial representation of a group $G$ is denoted by $\one_G.$ For $a \in \phi,$ we denote $\Sigma_a$ for the set of all inequivalent irreducible representations of $U_R$ that are constituents of $\ind_{\Stab(a)}^{U_R}(\one_{\Stab(a)}).$ In other words, $\Sigma_a$ can be written as
\begin{equation}
\label{def-Sigma}
\Sigma_a = \{ \rho \in \Irr(U_R) \mid \mathrm{Hom}_{U_R} ( \rho, \ind_{\Stab(a)}^{U_R}(\one_{\Stab(a)}) ) \neq 0  \}.
\end{equation} 
A $U_R$ representation $V$ is called {\em multiplicity-free} if $V$ decomposes as a 
direct sum of inequivalent irreducible representations each occurring with multiplicity at most one. 

\begin{defn} 
An element $a \in R \setminus U_R$ is called a {\em multiplicity-free non-unit} if  $\ind_{\Stab(a)}^{U_R}(\one_{\Stab(a)}) )$ is multiplicity-free as a $U_R$ representation. 
\end{defn}

The zero element of any ring is always a multiplicity-free non-unit of the ring. See Lemmas~\ref{lem:mfnu-commutative} and \ref{lem:induction-from-mirabolic} and the discussion following those for more examples of multiplicity-free non-units. The group $U_R$ acts on $R$ by conjugation. 
For any $r \in R,$ we call the set $C_r = \{ u r u^{-1} \mid u \in U_R \}$ as the {\em similarity class} of $r.$ If $r \in U_R,$ we will also use the term conjugacy class for $C_r.$
For any subset $S$ of $R,$ 
let $C_r S = \{ ts \mid t \in C_r \text{ and } s \in S \} .$ 
  Let $\psi$ be a fixed set of similarity class representatives of $R$ and $\psi^\times$ be a subset of $\psi$ consisting of similarity class representatives of $U_R.$ 
For $a \in \phi,$ we define the set $F_a$ as follows. 
 \begin{equation}
 \label{definition-of-Fa}
  F_a = \{ x \in \psi \mid (C_x S_a) \cap S_a \neq \emptyset \}.
 \end{equation}

Let $\mathbb{C}[R]$ be the complex vector space obtained by considering the formal basis $\{ e_r \mid r \in R \}.$ Further, define multiplication on basis elements of $\mathbb C[R]$ by the following. 
 \[
 e_r . e_s = e_{rs},\,\, \mathrm{for\,\,all} \,\,r, s \in R. 
 \]
By extending this multiplication linearly to all elements of $\mathbb C[R]$ we obtain that $\mathbb C[R]$ is a finite dimensional associative algebra which is not necessarily commutative. Every element of $R$ belongs to some principal ideal, and therefore $R = \cup_{a \in \phi} I_a.$ Further it is easy to see that $S_a = I_a \setminus \sum_{I_b \subsetneq I_a} I_b.$ This gives the decomposition
 \begin{equation}
 \mathbb C[R] = \underset{a \in \phi}{\oplus} \mathbb C[S_a],
 \end{equation}
where $\mathbb C[S_a]$ is a subspace of $\mathbb C[R]$ with formal basis $\{e_r \mid r \in S_a  \}.$ Similarly the spaces $\mathbb C[I_a]$ and $\mathbb C[U_R]$ are defined, where the latter coincides with the usual group algebra of $U_R.$ 
We have the following relation between $U_R$ and the elements of $S_a$

\begin{Lemma}[{\cite{steinberg-MO-2017}, 
\cite[Appendix A]{ayyer-steinberg-2017}}]
\label{lem:generation of Sa}
For any $x, y \in S_a,$ there exists $u \in U_R$ such that $u x = y .$ 
\end{Lemma}

In particular, Lemma~\ref{lem:generation of Sa} implies that the group $U_R$ acts transitively on $S_a$ by permuting its elements.  Therefore the set $S_a$ is in bijective correspondence with the set of left coset representatives of $U_R/\Stab(a)$ via $u \mapsto ua.$ This bijection preserves the $U_R$-action. Therefore as representations of $U_R,$ we have $\mathbb C[S_a] \cong  \ind_{\Stab(a)}^{U_R} (\one_{\Stab(a)}).$

 Given a representation $\rho: U_R \rightarrow \mathrm{GL}(V)$ of $U_R,$ we define a representation $\tilde{\rho}:\mathbb C[U_R] \rightarrow \mathrm{End}(V)$ of group algebra $\mathbb C[U_R]$ by $\tilde{\rho}\left( \sum_g a_g e_g \right) = \sum_g a_g \rho(g).$ This endows $V$ with a $\mathbb C[U_R]$-module structure. It is easy to see that $\rho$ is an irreducible representation of $U_R$ if and only if $V$ is an irreducible $\mathbb C[U_R]$-module under the above defined action of $\mathbb C[U_R].$ 
We are now in a position to state our result regarding the spectrum of the transition matrix $B_R.$

\begin{thm}
\label{thm:spectrum-general} 
Let $\phi' = \{ a \in \phi \mid a \in U_R \text{ or $a$ is a multiplicity-free non-unit}  \}.$ For every $a \in \phi',$ the following results hold.

\begin{itemize}
\item For every $x \in F_a,$ there exists a class function $\mathcal{F}_{x,a} \in \mathbb C[U_R]$ such that  for every $\rho \in \Sigma_a,$ we obtain an eigenvalue $\lambda_{\rho}$ of $B_R$ given by, 
\[
\lambda_{\rho}^a = \sum_{x \in F_a} Q(x) \frac{\mathrm{Tr}(\tilde \rho(\mathcal{F}_{x,a}))}{\dim(\rho)}.
\]

\item 
For every $\rho \in \Sigma_a$ the algebraic multiplicity, $m(\lambda_\rho^a)$ of $\lambda_{\rho}^a$ for $\rho \in \Sigma_a,$ satisfies
\[
m(\lambda_\rho^a) \geq  \sum_{\{ b \in \phi' \mid \lambda_\rho^a = \lambda_\rho^b \}} \dim(\rho),
\]
for any multiplcity free non-unit $a \in \phi$ and $m(\lambda_\rho^a) \geq \dim(\rho)^2$ for an invertible $a \in \phi.$
\end{itemize}

\end{thm}

\begin{proof} 
 It is well known that eigenvalues of $B_R$ are the same as that of the operator of the semigroup algebra $\mathbb C[R]$ obtained by multiplying on the left by $\sum_{x \in R} Q(x) x$ (see \cite[Section 7]{brown-2000}, for example). Further, by the definition of $Q(x),$ we have $\sum_{x \in R} Q(x) x = \sum_{x \in \psi} Q(x) \left( \sum_{v \in C_x}v \right).$ For $x \in \psi,$ define $\mathcal{H}_x := \sum_{v \in C_x} v$ and consider it as an operator on $\mathbb C[R]$ by left-multiplication. 
We restrict the action of $\mathcal{H}_x$ to $\mathbb C[I_a]$ and also consider the action of $\mathcal{H}_x$ on $\mathbb C[I_a] \setminus  \sum_{I_b \subsetneq I_a} \mathbb C[I_b].$
 We observe that the natural mapping of generators to generators give isomorphisms $\mathbb C[I_a] \setminus  \sum_{I_b \subsetneq I_a} \mathbb C[I_b] \cong \mathbb{C}[S_a]$ and $\mathbb C[I_a] \cong \mathbb{C}[S_a] \oplus \sum_{I_b \subsetneq I_a} \mathbb C[I_b] $ as $U_R$-spaces. We use these isomorphisms to define an action of $\mathcal H_x$ on $\mathbb C[S_a]$ for every $a \in \phi$ and denote this by $\mathcal{H}_{x,a}.$ We note that $\mathcal{H}_{x} u' = u' \mathcal{H}_{x} $ for all $u' \in U_R$ and therefore, the action of  $\mathcal{H}_{x,a}$ on $\mathbb C[S_a]$ is $U_R$-linear.

We define a partial order on $\phi$ by $a \leq b$ if and only if $I_b \subseteq I_a.$
For each $a \in \phi,$ we fix an ordered basis $\mathcal{B}_a$ of $\mathbb C[S_a].$ We also fix an ordered basis for $\mathbb C[R]$ where elements of $\mathcal{B}_b$ appear before $\mathcal{B}_a$ if $a \leq b.$ We note that for every $a \in \phi$ and every $z \in \mathbb C[S_a],$ we have $\sum_{x \in \psi } \mathcal{H}_x z =  \mathcal{H}'_{x,a} z +\omega$ for some $\mathcal{H}'_{x,a} \in \mathbb C[U_R]$ and $\omega \in \sum_{I_b \subsetneq I_a} \mathbb C[I_b].$ Thus, the matrix of $\mathcal{H}_x$ in the above basis is block upper-triangular. Our goal is to show that for $a \in \phi'$ there exists a choice of basis of $\mathcal{B}_a$  such that the diagonal blocks are actually diagonal with the desired entries. We consider the cases of $a$ is a unit and $a$ is a multiplicity-free non-unit separately.

\noindent { \bf Case~1: $a$ is a unit. } 
Here, $F_a = \psi^\times$ and therefore $\sum_{x \in \psi } \mathcal{H}_x z = \mathcal{F}_{x,a} (z) +\omega,$ where $\mathcal{F}_{x,a} = \sum_{x \in \psi^\times} Q(x)\left( \sum_{v \in C_x}v \right) \in \mathbb C[U_R]$ is a class function and $\omega \in \sum_{I_b \subsetneq I_a} \mathbb C[I_b].$ Being a class function,  $\mathcal{F}_{x,a} $ acts as a scalar on each irreducible constituent of $\mathbb C[S_a]$ and the result follows in this case. Furthermore, in this case $a$ is a unit so $\mathbb C[S_a] \cong \mathbb C[U_R]$ is the regular representation of
$U_R$ so each eigenvalue $\lambda_\rho^a$ occurs with multiplicity greater than equal to $\dim(\rho)^2.$ 
	 
\noindent { \bf Case~2: $a$ is a multiplicity-free non-unit. } 
For this case, there exists a decomposition of $\mathbb C[S_a] \cong V_1 \oplus V_2 \oplus \cdots \oplus V_t$ into $U_R$-irreducible constituents such that $V_i \ncong V_j.$ Let $d_i = \dim(V_i)$ for all $1 \leq i \leq t.$ Then due to $U_R$-linearity we have $\mathcal{H}_{x,a}(V_i ) \subseteq V_i$ for all $1 \leq i \leq m$ and by Schur's lemma, the action of $\mathcal{H}_{x,a}$ on each $V_i$ is multiplication by a scalar. Therefore, by abstract representation theory, there exists a class function $\mathcal{F}_{x,a} \in \mathbb C[U_R]$ such that its action coincides with that of $\mathcal{H}_{x,a}$ on $\mathbb C[S_a].$ Consider an ordered basis $\mathcal{B}_a = \{ v_{1,1}, \cdots v_{1,d_1}, v_{2,1}, v_{2,2}, \ldots, v_{t,d_t}  \}  $ of $\mathbb C[S_a]$ such that $v_{i,j} \in V_i$ for $1 \leq j \leq d_j.$ The function $\mathcal{F}_{x, a} \in \mathbb C[U_R]$ is a class function. 
Hence, the matrix of $\mathcal{F}_{x, a}$ with respect to $\mathcal{B}_a,$ as an endomorphism of $\mathbb C[S_a],$ is a diagonal matrix.
On each constituent $\rho \in \Sigma_a$ of  $\ind_{\Stab(a)}^{U_R}(\one_{\Stab(a)}),$ $\mathcal{F}_{x, a}$ acts as multiplication by 
the scalar $ \frac{\mathrm{Tr}(\tilde \rho(\mathcal{F}_{x,a}))}{\dim(\rho)}.$ 
Therefore, this scalar is an eigenvalue of $\mathcal{F}_{x, a}$
with multiplicity equal to $\dim(\rho).$  
Combining this result with the fact that $\sum_{x \in R} Q(x) = \sum_{x \in \psi} Q(x) \mathcal{H}_x,$ we obtain our result. 
\end{proof} 

\begin{Remark} 
\label{rem:phi=phi'}
For any ring that satisfies $\phi' = \phi,$ Theorem~\ref{thm:spectrum-general} gives the complete spectrum of $B_R.$  
\end{Remark}

\begin{Remark} 
For $a=0 \in \phi',$ we obtain the eigenvalue $1$ of $B_R.$
Since $|U_R \cap \phi'| = 1,$ we also obtain $|U_R|$-many eigenvalues of $B_R$ counted with multiplicity by Theorem~\ref{thm:spectrum-general}. 
Further, every other $a \in \phi'$ will give more eigenvalues of $B_R.$ 
\end{Remark} 

The following result is an easy application of Frobenius reciprocity and the fact that every irreducible representation of an abelian group is one-dimensional.

\begin{Lemma}
\label{lem:mfnu-commutative}
Let $G$ be a finite abelian group and $H$ be a subgroup of $G.$ Then the representation $\mathrm{Ind}_H^G(\one_H)$ is multiplicity-free. 
\end{Lemma}

Lemma~\ref{lem:mfnu-commutative} implies that $\phi' = \phi$ for finite commutative rings. Therefore, all eigenvalues of $B_R$ can be computed by Theorem~\ref{thm:spectrum-general} and Remark~\ref{rem:phi=phi'}. This result was obtained in an earlier work~\cite{AS-2016}. 

\begin{Lemma}
\label{lem:induction-from-mirabolic}
 Let $G = \GL_n(\mathbb F_q)$ and $P $ be the subgroup of $G$ consisting of
  all matrices with last row $(0, 0, \cdots, 0, 1).$ Let $H$ be a subgroup of $G$ conjugate to $P.$ Then $\mathrm{Ind}_{H}^G(\one_{H})$ is multiplicity-free. 
\end{Lemma}

\begin{proof} 
It is well known (see, for example, \cite[Section~13.5]{zelevinsky-1981}) that the induced representation $\mathrm{Ind}_{P}^G(\one_{P})$ is multiplicity-free. Let $g \in G$ such  that $gPg^{-1} = H$ and for an irreducible representation $\rho$ of $G,$ the conjugate representation $\rho^g$ is defined by $\rho^g(x) = \rho(g^{-1}xg).$ 
Then, an irreducible representation $\rho$ is a constituent of $\mathrm{Ind}_{P}^G(\one_{P})$ if and only if $\rho^g$ is a constituent of $\mathrm{Ind}_{H}^G(\one_{H} ).$ Moreover, $\mathrm{Hom}_G\left(\rho, \mathrm{Ind}_{P}^G(\one_{P})  \right)= \mathrm{Hom}_G \left( \rho^g, \mathrm{Ind}_{H}^G(\one_{H} ) \right).$ 
Therefore $\mathrm{Ind}_{H}^G(\one_{H})$ is multiplicity-free for any subgroup $H$ conjugate to $P.$
\end{proof}

The way Lemma~\ref{lem:induction-from-mirabolic} is useful for $R = M_n(\mathbb F_q)$ is as follows: Consider a matrix $a= (a_{i,j})_{1 \leq i,j \leq n} \in R,$ 
such that $a_{n,n} = 1$ and $a_{i,j} = 0$ otherwise. 
Then $\mathrm{LStab}(a) = P.$ 
For any matrix $r \in R$ of rank one, there exists invertible matrices $u, v \in U_R$ such that $r = u a v,$ i.e. $r$ and $a$ are equivalent matrices.  
Therefore, $\mathrm{LStab}(r)$ is conjugate to the group $P.$ By Lemma~\ref{lem:induction-from-mirabolic}, all these rank one elements are multiplicity-free non-units of $R.$ It is easy to see that there are many element in $\phi$ of rank $1,$ and all of these will give eigenvalues of $B_R$ by Theorem~\ref{thm:spectrum-general}.  
As remarked earlier, the zero element of any ring is easily seen to be a multiplicity-free non-unit. 
For $R = M_2(\mathbb F_q),$ every element of $\phi$ is either a zero element or of rank one or a unit. Therefore, by the above discussion, $\phi = \phi'$ for $R = M_2(\mathbb F_q).$ 
We use this to give explicit eigenvalues of $B_R$ along with their multiplicities for the ring $R = M_2(\mathbb F_q)$ for odd $q$ in Section~\ref{sec:gl2}. 

We note that for the ring $R = \mathbb B_2(\mathbb F_q) = \{ x \in M_2(\mathbb F_q) \mid x_{2,1} = 0   \},$ it can be shown by direct computations that $\phi = \phi'.$ Therefore, the complete spectrum of $B_R$ can be obtained in this case as well.

\section{Stationary distribution}
\label{sec:statdist}

We now give a recursive formula for the stationary distribution $\pi$ of the chain. The formula is a special case of the one for arbitrary Markov chains~\cite{rhodes-schilling-2017} and is very similar in spirit to the commutative case~\cite[Theorem 2.4]{AS-2016}.

\begin{prop}
\label{prop:equalonSa}
For $a,b \in R$ such that $S_b = S_a,$ we have that $\pi(a) = \pi(b).$
\end{prop}

\begin{proof}
There exists an element $u \in U_R$ such that $u a = b$ by Lemma~\ref{lem:generation of Sa}. Therefore, left multiplication by $u$ is an inner automorphism of $R$ which takes $a$ to $b.$ Moreover, for any $c,d \in R,$
\[
B_R(uc,ud) = \sum_{\substack{x \in R \\ xuc = ud}} Q(x)
= \sum_{\substack{y \in R \\ yc = d}} Q(u y u^{-1})
= \sum_{\substack{y \in R \\ yc = d}} Q(y) = B_R(c,d).
\]
Therefore, all transition rates are unchanged, and the automorphism causes only a relabelling of rows and columns of $M_R.$
\end{proof}

 For any $x,y \in R,$ let $R_{x,y} = \{ r \mid ry = x\}.$ Then $R_{x,y} \neq \emptyset$ if and only if $I_{x} \subseteq I_y.$ Further, in case $R_{x, y} \neq \emptyset$ then we must have $|R_{x,y}| = |\ann_R(y)|.$ For any $x \in R,$ we use $U_x \subseteq U_R$ to denote the set of distinct coset representatives of $\Stab(x)$ in $U_R.$ In case $R$ is commutative the set $U_x$ has a group structure however this does not hold for the case of non-commutative ring $R.$  

We now state the main theorem of this section. Recall that $\phi$ is a set of generators of distinct principal ideals of $R.$

\begin{thm}
\label{thm:statdist}
Let $R$ be a finite ring. The stationary probability $\pi(x)$ for $x \in R$ in
$\X t$ is given by
\[
\pi(x) = \frac{\ds \frac{\alpha}{|R|} + (1-\alpha)
\sum_{y \in \phi, I_x \subsetneq I_y} \left( 
\sum_{u \in U_y, r \in R_{x,y}} Q(r u^{-1}) \right) \pi(y) }
{\ds 1 - (1-\alpha) \left( \sum_{u \in U_x, r \in R_{x,x}} Q(r u^{-1})  \right) }
\]
\end{thm}

\begin{proof}
The strategy of proof is essentially identical to that of the proof of Theorem 2.4 in \cite{AS-2016}, and we will be a lot more sketchy. 

The stationary distribution satisfies the master equation,
\[
\pi(x) = \sum_{y \in R} \mathbb{P}(y \to x) \pi(y).
\]
Now, $\mathbb{P}(y \to x) \geq \alpha/|R|$ for all $y \in R$ because of the addition transition. To keep track of when multiplicative transitions can occur, we look at the poset of principal ideals, and we see that
\[
\pi(x) = \frac{\alpha}{|R|} + (1-\alpha) \sum_{
\substack{y \in R \\ I_x \subseteq I_y}} B_R(y,x) \pi(y).
\]
Now, we split the sum on the right hand side depending on whether $I_y = I_x$ or not. Using Lemma~\ref{lem:generation of Sa} and Proposition~\ref{prop:equalonSa}, we find that
\[
\sum_{\substack{y \in R \\ I_x = I_y}} B_R(y,x) \pi(y)
= \pi(x) \sum_{\substack{u \in U_x \\ r \in R_{x,x}}} Q(r u^{-1}),
\]
and
\[
\sum_{\substack{y \in R \\ I_x \subsetneq I_y}} B_R(y,x) \pi(y)
= \sum_{\substack{y \in \phi \\ I_x \subsetneq I_y}} \pi(y)
\sum_{\substack{u \in U_y \\ r \in R_{x,y}}} Q(r u^{-1}).
\]
Substituting the last two identities in the previous sum gives us the desired result.
\end{proof}

As a corollary, we obtain the stationary distribution for $\Xu t.$ Somewhat surprisingly, the answer is identical to that of the commutative case (Corollary 2.5) in \cite{AS-2016}.

\begin{cor}
\label{cor:statdistunif}
Let $R$ be a finite ring. The stationary probability $\pi(x)$ for $x \in R$ in $\Xu t$ is given by
\[
\pi_U(x) = \frac{\alpha + (1-\alpha)
\sum_{y \in \phi, I_x \subsetneq I_y} |U_y| \; |\ann_R(y)| \; \pi_U(y) }
{ |R| - (1-\alpha) |U_x| \; |\ann_R(x)| }.
\]
\end{cor}

\begin{proof} 
This immediately follows from Theorem~\ref{thm:statdist} and the fact that  $|R_{x,y}| = |\ann_R(y)|$ for every $x, y \in R$ such that $I_x \subseteq I_y.$ 
\end{proof}

The formula for the stationary probability of units can then be derived immediately. 

\begin{cor}
\label{cor:statdistunif-units}
Let $R$ be a finite ring. The stationary probability $\pi(u)$ for $u \in U_R$ in $\Xu t$ is given by
\[
\pi_U(u) = \frac{\alpha}{n-u + u \alpha}.
\]
\end{cor}

The above result is identical to that obtained for commutative rings~\cite{AS-2016}.

\begin{example}
For the ring $M_2(\mathbb F_2),$ the multiplicative part of the transition matrix $B_R$
in the lexicographically ordered basis,
\begin{align*}
{\scriptstyle \left\{ \matt 0000, \matt 0001, \matt 0010, \matt 0011, \matt 0100, \matt 0101, 
\matt 0110, \matt 0111, \right. } \\
{\scriptstyle \left. \matt 1000, \matt 1001, \matt 1010, \matt 1011, \matt 1100, \matt 1101, 
\matt 1110, \matt 1111 \right\} },
\end{align*}
is given by
\[
{\scriptsize \frac{1}{16}
\left(
\begin{array}{cccccccccccccccc}
 16 & 0 & 0 & 0 & 0 & 0 & 0 & 0 & 0 & 0 & 0 & 0 & 0 & 0 & 0 & 0 \\
 4 & 4 & 0 & 0 & 4 & 4 & 0 & 0 & 0 & 0 & 0 & 0 & 0 & 0 & 0 & 0 \\
 4 & 0 & 4 & 0 & 0 & 0 & 0 & 0 & 4 & 0 & 4 & 0 & 0 & 0 & 0 & 0 \\
 4 & 0 & 0 & 4 & 0 & 0 & 0 & 0 & 0 & 0 & 0 & 0 & 4 & 0 & 0 & 4 \\
 4 & 4 & 0 & 0 & 4 & 4 & 0 & 0 & 0 & 0 & 0 & 0 & 0 & 0 & 0 & 0 \\
 4 & 4 & 0 & 0 & 4 & 4 & 0 & 0 & 0 & 0 & 0 & 0 & 0 & 0 & 0 & 0 \\
 1 & 1 & 1 & 1 & 1 & 1 & 1 & 1 & 1 & 1 & 1 & 1 & 1 & 1 & 1 & 1 \\
 1 & 1 & 1 & 1 & 1 & 1 & 1 & 1 & 1 & 1 & 1 & 1 & 1 & 1 & 1 & 1 \\
 4 & 0 & 4 & 0 & 0 & 0 & 0 & 0 & 4 & 0 & 4 & 0 & 0 & 0 & 0 & 0 \\
 1 & 1 & 1 & 1 & 1 & 1 & 1 & 1 & 1 & 1 & 1 & 1 & 1 & 1 & 1 & 1 \\
 4 & 0 & 4 & 0 & 0 & 0 & 0 & 0 & 4 & 0 & 4 & 0 & 0 & 0 & 0 & 0 \\
 1 & 1 & 1 & 1 & 1 & 1 & 1 & 1 & 1 & 1 & 1 & 1 & 1 & 1 & 1 & 1 \\
 4 & 0 & 0 & 4 & 0 & 0 & 0 & 0 & 0 & 0 & 0 & 0 & 4 & 0 & 0 & 4 \\
 1 & 1 & 1 & 1 & 1 & 1 & 1 & 1 & 1 & 1 & 1 & 1 & 1 & 1 & 1 & 1 \\
 1 & 1 & 1 & 1 & 1 & 1 & 1 & 1 & 1 & 1 & 1 & 1 & 1 & 1 & 1 & 1 \\
 4 & 0 & 0 & 4 & 0 & 0 & 0 & 0 & 0 & 0 & 0 & 0 & 4 & 0 & 0 & 4 \\
\end{array}
\right)}.
\]
The stationary probability of the units is 
\begin{align*}
\pi \matt 0110 &= \pi \matt 0111 = \pi \matt 1001 = \\
\pi \matt 1011 &= \pi \matt 1101 = \pi \matt 1110 = 
\frac{\alpha }{2 (3 \alpha +5)},
\end{align*}
of the nonzero non-units is
\begin{align*}
\pi \matt 0001 &= \pi \matt 0010 = \pi  \matt 0011 = \pi  \matt 0100 =  \pi  \matt 0101 =\\
\pi \matt 1000 &= \pi \matt 1010 = \pi \matt 1100 = \pi \matt 1111 = \frac{2 \alpha }{(3 \alpha +1) (3 \alpha +5)}, 
\end{align*}
and of zero is
\begin{align*}
\pi \matt 0000 = \frac{5 - 3 \alpha}{(3 \alpha +1) (3 \alpha +5)}.
\end{align*}
\end{example}

\section{Mixing time}
\label{sec:mixing}

We will now prove that $\X t$ on an arbitrary finite ring $R$ mixes in finite time using a coupling argument. The argument is identical to that for commutative chains in \cite{AS-2016}, but we repeat it for the sake of completeness.

We begin with the basic definitions.
The {\em (total variation) distance} between probability distributions $\mu$ and $\nu$ on the same space $\Omega$ is given by
\[
||\mu - \nu ||_{\text{TV}} = \frac{1}{2} \sum_{x \in \Omega} |\mu(x) - \nu(x)|.
\]
Let $M_R$ be the transition matrix of the chain $\X t,$ and $\pi$ be the stationary distribution. Then the distance between the distribution of the chain at time $t$ and the stationary distribution is given by
\[
d(t) = \max_{x \in R} ||M_R^t(x,\cdot) - \pi ||_{\text{TV}}.
\]
For a fixed $\epsilon < 1/2,$ the {\em mixing time} of the Markov chain is
\[
t_{\text{mix}}(\epsilon) = \min \{t \mid d(t) \leq \epsilon \}.
\]
A {\em (Markovian) coupling} of our Markov chains $(X_t, X'_t)_{t \in \mathbb{Z}_{\geq 0}}$ is a Markov chain on $R \times R$ such that the marginals $X_t$ and $X'_t$ are themselves Markov chains on $R$ with transition matrix $M_R.$
For a coupling, let $t_{\text{couple}}$ be the first time the chains meet, known as the {\em coupling time}. The mixing time can be bounded by the coupling time using the following result~\cite[Corollary 5.3]{LevinPeresWilmer}
\[
d(t) \leq \max_{x,x' \in R} \mathbb{P}_{x,x'}(t_{\text{couple}} > t).
\]

\begin{thm}
\label{thm:mixing}
The mixing time of the Markov chain $\X t$ on a finite ring $R$ is bounded above by the absolute constant
\[
t_{\text{mix}}(\epsilon) \leq \frac{\log \epsilon}{\log(1-\alpha)} + 1.
\]
\end{thm}

\begin{proof}
We will couple two runs of the Markov chain starting at $x,x' \in R$ as follows. At each time we toss an independent coin with heads probability $\alpha.$ If the coin lands Tails, both chains choose independent samples from the distribution $Q$ and multiply with the chosen element. If it lands Heads, we ensure the two chains meet. To be precise, if the chains are at positions $y,y'$ at some time $t,$ then we choose a uniformly random element $r \in R$ (say). We then add $r$ to $y$ and add $r + y - y'$ to $y'.$ Since $r$ is uniformly random, so is $r + y - y',$ and the marginals move according the original chain.

The coupling time is then the usual waiting time distribution, which is a geometric random variable with success probability $\alpha.$ Thus, the probability that the two chains have not met up to time $t$ is bounded above by $(1-\alpha)^t.$ We add an extra step in case $\alpha$ is very close to $1$ so that $t_{\text{mix}}(\epsilon)$ does not become smaller than one. The result then follows.
\end{proof}

\section{The matrix ring $\mathrm{M}_2(\mathbb F_q)$}
\label{sec:gl2}
Let $R = M_2(\mathbb F_q)$ be the ring of $2 \times 2$ matrices with entries over the finite field $\mathbb F_q,$ where $q > 2.$ For $q=2,$ the description of conjugacy classes of $\GL_2(\mathbb F_q)$ is slightly different, but the computations are similar.

Recall that $U_R = \mathrm{GL}_2(\mathbb F_q)$ is the group of $2 \times 2$ invertible with entries from $\mathbb F_q$ and $|U_R| = (q^2-1)(q^2-q).$ To avoid any confusion, in this section we will use uppercase letters to denote elements of $R$ and lowercase letters for the field $\mathbb F_q.$ 

The set
 \[ 
\phi = \left \{ \matt 1001, \matt 0100, \matt 0000 \right \}
\bigsqcup \left \{  \matt 1{z}00 \right \}_{z \in \mathbb F_q}
\]
is a set of representatives of distinct principal ideals of $R.$ The set
\[
\psi = \left \{ \matt x00x \right \}_{x \in \mathbb F_q} \bigsqcup
\left \{ \matt x01x \right \}_{x \in \mathbb F_q} \bigsqcup
\left \{ \matt x00y \right \}_{x \neq y \in \mathbb F_q} \bigsqcup
\left \{ \matt {\alpha}00{\bar{\alpha}} \right \}_{\alpha \in \mathbb F_{q^2} \setminus \mathbb F_q} 
\] 
is a set of representatives of similarity classes of $M_2(\mathbb F_q).$ Further, we say a similarity class is {\em invertible} if it is contained in $\GL_2(\mathbb F_q).$ Otherwise we call a similarity class {\em non-invertible}. We note that this does not depend on the similarity class representative. The set of non-invertible similarity classes, denoted $\psi^{0},$ is given by 
\[
\psi^{0} = \left \{  \matt 0000, \matt 0010 \right \} \bigsqcup
\left \{ \matt x000 \right \}_{x \in \mathbb F_q^\times} .
\]
We denote the set of invertible classes as $\psi^\times = \psi \setminus \psi^{0}.$  Next, for each $A \in \phi,$ we describe below the set of generators of the principal ideal generated by $A,$ denoted by $S_A,$ and the (left) stabilizer of $A$ in $U_R,$ denoted by $\Stab(A).$ 
Recall that $\Stab(A)$ is the set of elements $X \in U_R$ such that $XA = A.$ 

\begin{table}[htbp!]
\[
\begin{array}{|c|c|c|c|}
\hline
A \in \phi & S_A & \Stab(A) & F_A \cap \psi^0 \\ 
\hline 
\matt 1001 & \GL_2(\mathbb F_q) & \matt 1001 & \emptyset \\ 
\hline
\matt 0100 & \left \{\matt 0a0b \right \}_{(a,b) \neq (0,0)} & \left \{ \matt 1y0w \right \}_{\substack{y \in \mathbb{F}_q \\ w \in \mathbb F_q^\times}}
&  \matt 0010, \left \{ \matt t000 \right \}_{t \in \mathbb F_q^\times} \\ 
\hline
\matt 1{z}00_{z \in \mathbb F_q} & \left \{ \matt a{za}b{zb}  \right \}_{(a,b) \neq (0,0)} 
& \left \{ \matt 1y0w \right \}_{\substack{y \in \mathbb{F}_q \\ w \in \mathbb F_q^\times}} &   \matt 0010, \left \{ \matt t000 \right \}_{t \in \mathbb F_q^\times} \\
\hline
\matt 0000 & \matt 0000  & \GL_2(\mathbb F_q)  & \psi^0  \\
\hline
\end{array} 
\]
\caption{For each generator $A$ of a principal ideal in $\mring q,$ a description of the sets $S_A,$ $\Stab(A)$ and $F_A \cap \psi^0.$}
\label{tab:stab}
\end{table}

\subsection{Description of $F_A$ and $\mathcal{F}_{X, A}$ }
\label{sec:description-of-F-x-a} 
We now describe $F_A,$ defined in \eqref{definition-of-Fa}, for every $A \in \phi.$ For any $x \in \psi$ such that $s$ is invertible, we have $x S_A \subseteq S_A.$ Therefore $\psi^\times \subseteq F_A$  for each $A \in \phi.$ It remains to determine $F_A \cap \psi^0.$ We give the representatives of similarity classes that are contained in $F_A \cap \psi^0$ in Table~\ref{tab:stab} for every $A \in \phi.$ 

Our next step is, for each $A \in \phi$ and $X \in F_A,$ to identify $\mathcal{F}_{X, A}$ satisfying the hypothesis of Theorem~\ref{thm:spectrum-general}. To that end, 
consider the operator $\mathcal{H}_{X, A}= \sum_{V \in U_R} VXV^{-1} : \mathbb C[I_A] \rightarrow \mathbb C[S_A]$ and $\mathcal{H}'_{X,A}: \mathbb C[S_A] \rightarrow \mathbb C[S_A],$ obtained by composing $\mathcal{H}_A$ with the natural inclusion and projection.  
For case $X \in F_A \cap \psi^\times,$ the function $\mathcal{F}_{X, A} = \mathcal{H}'_{X, A}$ itself works. 
Therefore, we now focus on the case where $X \in F_A \cap \psi^0,$ and we explicitly describe $\mathcal{F}_{X, A}.$ The existence of these was proved in Theorem~\ref{thm:spectrum-general}. For $A = \matt 0000,$ the function $\mathcal{F}_{X, A} = |C_X|e_{\left(\begin{smallmatrix} 1 & 0 \\ 0 & 1  \end{smallmatrix}\right)}$ itself works for all $X \in F_A \cap \psi^0.$ We are left to deal with the cases for all $A \in \left \{\matt 0100 \right \}
\bigsqcup \left \{  \matt 1{z}00 \right \}_{z \in \mathbb F_q}.$

For simplicity, we use the notation,
\[ 
Y_t =\matt t000, \quad \text{for $t \in \mathbb F_q^\times$ and} 
\quad Y_0 = \matt 0010. 
\]
We obtain the following by direct computations and these work uniformly for all such $A.$
\begin{equation}
\label{eq:existence-of-operators}
\begin{split}
\mathcal{F}_{Y_t, A} & = \sum_{v \in C_{ 
\left(\begin{smallmatrix} t & 1 \\ 0 & t  \end{smallmatrix}\right)}
} e_v + e_{ \left(\begin{smallmatrix} t & 0 \\ 0 & t  \end{smallmatrix}\right) },  \\ 
\mathcal{F}_{Y_0, A} & = \sum_{v \in C_{ 
\left(\begin{smallmatrix} 1 & 1 \\ 0 & 1  \end{smallmatrix}\right)}
} e_v - (q-1) e_{ \left(\begin{smallmatrix} 1 & 0 \\ 0 & 1  \end{smallmatrix}\right)}.
\end{split}
\end{equation}

\subsection{Description of $\Sigma_A$} 
Next, for every $A \in \phi,$ we describe $\Sigma_A,$ defined in \eqref{def-Sigma}. 
Thus $\Sigma_A$ consists of inequivalent irreducible constituents of $\mathrm{Ind}_{\Stab(A)}^{U_R} (\one_{\Stab(A)}).$ 

In one extreme case, $\Stab(A) = \GL_2(\mathbb F_q),$ we have $U_R = \Stab(a)$ and therefore $\mathrm{Ind}_{\Stab(A)}^{U_R} (\one_{\Stab(a)})$ is the trivial representation of $\GL_2(\mathbb F_q).$ In the other extreme, $\Stab(A) = \matt 1001$ and $\mathrm{Ind}_{\Stab(A)}^{U_R} (\one_{\Stab(a)})$ is the regular representation of $\GL_2(\mathbb F_q)$ and therefore $\Sigma_A$ in this set consists of all inequivalent irreducible representations of $\GL_2(\mathbb F_q).$ For the reader's convenience, we have given a  description of all irreducible representations of $\GL_2(\mathbb F_q)$ along with their character values in Appendix~\ref{sec:appendix}.

Let $P$ be the subgroup of $\GL_2(\mathbb F_q)$ consisting of matrices of the form $\matt 1y0w,$ where $w \in \mathbb F_q^\times.$ Then, from Table~\ref{tab:stab}, we see that $P$ is $\Stab(A)$ for some $A \in \phi.$ So, we need to describe to describe the irreducible constituents of $\mathrm{Ind}_P^{\GL_2(\mathbb F_q)}(\one_P).$ The trivial representation of the group $\mathbb F_q^\times$ is denoted by $\one_q$ 
for simplification. We use the notation of Appendix~\ref{sec:appendix} to give this description.  

\begin{prop} 
The space $\mathrm{Ind}_P^{\GL_2(\mathbb F_q)}(\one_P)$ has the following decomposition into $\GL_2(\mathbb F_q)$-irreducible
 constituents.
\begin{equation}
\label{irreducible-decomposition-for-P}
\mathrm{Ind}_P^{\GL_2(\mathbb F_q)}(\one_P) \cong \one_{\GL_2(\mathbb F_q)} \oplus \rhoq \oplus_{\chi \in \widehat{\mathbb F_q^\times} \setminus \one_q} \rho_{\chi, \one_q},
\end{equation} 
where $\rhoq$ and $\rho_{\chi,\one_q}$ are given in Table~\ref{tab:chartable}.
\end{prop} 

\begin{proof} 
Let $\mathbb S$ be the set of all irreducible constituents listed on the right side of (\ref{irreducible-decomposition-for-P}). 

We first note that both $\mathrm{Ind}_P^G(\one_P)$ and $\one_G \oplus \rhoq \oplus_{\chi \in \widehat{\mathbb F_q^\times} \setminus \one_q} \rho_{\chi, \one_q}$ have dimension $(q^2 -1).$ 
Thus, to complete the proof it is enough to show that any $\rho \in \mathbb S$ is indeed a constituent of $\ind_P^{\GL_2(\mathbb F_q)}(\one_P).$ By Frobenius reciprocity, this is equivalent to proving that the restriction of $\rho$ to $P$ for each $\rho \in \mathbb S$ has a $P$-fixed vector, that is $\mathrm{Hom}_P(Res^{\GL_2(\mathbb F_q)}_P(\rho), \one_P ) \neq 0,$ where $\mathrm{Res}^{\GL_2(\mathbb F_q)}_P(\rho)$ denotes the restriction of $\rho$ to $P.$ The space $\mathrm{Hom}_{G}(\rho_1, \rho_2)$ of intertwiners between two representations $\rho_1$ and $\rho_2$ of $G$ is given by 
\begin{eqnarray} 
\label{eq:intertwiner} 
\dim_{\mathbb C} (\mathrm{Hom}_{G}(\rho_1, \rho_2)) & = &  \frac{1}{|G|} \sum_{x \in G} (\mathrm{Tr}(\rho_1(x)) \overline{\mathrm{Tr}(\rho_2(x))}.
\end{eqnarray} 
Therefore, $\mathrm{Hom}_P(Res^{\GL_2(\mathbb F_q)}_P(\rho), \one_P ) \neq 0$ for any $\rho \in \mathbb S$  is equivalent to the fact that $\frac{1}{|P|} \sum_{x \in P} \mathrm{Tr}(\rho(x)) \neq  0.$ 
We further note that 
\begin{equation} 
\sum_{x \in P} \mathrm{Tr}\left (\rho(x) \right ) = \mathrm{Tr}\left (\rho \matt 1001 \right ) + (q-1) \mathrm{Tr}\left (\rho \matt 1101 \right ) \nonumber + (q) \sum_{x \in \mathbb F_q^\times \setminus 1} \mathrm{Tr}\left (\rho\matt x001  \right). 
\end{equation} 
Using the character values from Table~\ref{tab:chartable}, we obtain the following for each $\rho \in \mathbb S$:

\[
\begin{array}{lll}
\ds\rho = \one_{\GL_2(\mathbb F_q)}: & \ds\quad  \sum_{x \in P} \mathrm{Tr}\left (\rho(x) \right ) &=  1 + (q-1) + q(q-2) = q(q-1). \\[0.5cm]
\ds\rho = \rhoq:  & \ds \quad \sum_{x \in P} \mathrm{Tr}\left (\rho(x) \right )  &= q + q(q-2) = q(q-1). \\[0.4cm]
\ds\rho = \rho_{\chi, \one_q}: & \ds \quad \sum_{x \in P} \mathrm{Tr}\left (\rho(x) \right ) & = \ds(q+1) + (q-1) + q \left( \sum_{x \in \mathbb F_q^\times\setminus 1} (\chi(x) + 1) \right) \\[0.4cm]
\ds &&=  \ds 2q + q \left(\sum_{x \in \mathbb F_q^\times} \chi(x) \right) + q(q-3). 
\end{array}
\]

For every $\mu \in \widehat{\mathbb F_q^\times} \setminus \one_q,$ we have $\mathrm{Hom}_{\mathbb F_q^\times}( \mu, \one_q) = 0.$ From (\ref{eq:intertwiner}), it is easy to see that $\sum_{x \in \mathbb F_q^\times} \mu(x) = 0$ for every $\mu \in \widehat{\mathbb F_q^\times} \setminus \one_q.$ Therefore for every $\rho \in \mathbb S,$ we have $\frac{1}{|P|} \sum_{x \in P} \mathrm{Tr}(\rho(x)) \neq  0.$ This proves the equivalence given in (\ref{irreducible-decomposition-for-P}). 
\end{proof} 

We summarize all of this information of $F_A$ and $\Sigma_A$ for $A \in \phi$ in Table~\ref{tab:FA}.

\begin{table}[htbp!]
\[
\begin{array}{|c|c|c|c|}
\hline
S. No. & A \in \phi & F_A  & \Sigma_A  \\
\hline
1. &  \matt 1001 & \psi^\times & \mathrm{Irr}(\GL_2(\mathbb F_q)) \\
\hline
2. & \matt 0100 &  \psi^\times \cup Y_0 \cup Y_t & \one_{\GL_2(\mathbb F_q)} \cup  \rhoq \underset{{\chi \in \widehat{\mathbb F_q^\times} \setminus \one_q}}{\cup} \rho_{\chi, \one_q} \\
\hline
3. & \left\{ \matt 1{z}00 \right\}_{z\in \mathbb F_q} &  \psi^\times \cup Y_0 \cup Y_t  & \one_{\GL_2(\mathbb F_q)} \cup  \rhoq \underset{{\chi \in \widehat{\mathbb F_q^\times} \setminus \one_q}}{\cup} \rho_{\chi, \one_q} \\
\hline
4. &  \matt 0000  & \psi & \one_{\GL_2(\mathbb F_q)} \\
\hline
\end{array}
\]
\caption{The sets $F_A$  and $\Sigma_A$ for $A \in \phi.$}
\label{tab:FA}
\end{table}

Using Table~\ref{tab:FA} along with the description of $\mathcal{F}_{X,A}$ as given in Section~\ref{sec:description-of-F-x-a}, the character table in Table~\ref{character-table} and Theorem~\ref{thm:spectrum-general}, we obtain a complete description of the eigenvalues of $B_R.$ We now describe the eigenvalues. Recall from Theorem~\ref{thm:spectrum-general}, that for each $A \in \phi$ and $\rho \in \Sigma_A,$ the eigenvalue $\lambda_\rho$ is given by 
\[
\lambda_\rho = \sum_{X \in F_A} Q(X) \frac{\mathrm{Tr}(\tilde \rho(\mathcal{F}_{X,A}))}{\dim(\rho)}. 
\] 

\begin{itemize} 

\item[(a)] For $A = \matt 1001$ and $\rho \in \Sigma_A,$ we have 
\[
\lambda_\rho = \sum_{X \in \psi^\times } Q(X) \, |C_X|\, 
\chi_\rho(X),
\]
and the multiplicity of each $\lambda_{\rho}$ in this case is 
$\dim(\rho)^2.$

\item[(b)] For any $A \in \left \{\matt 0100 \right \}
\bigsqcup \left \{  \matt 1{z}00 \right \}_{z \in \mathbb F_q}$ and $\rho \in \Sigma_A,$ we have; 
\begin{equation}
\begin{split}
\label{eq:e-genvalues1}
\dim(\rho)(\lambda_\rho)  =&  \sum_{X \in \psi^\times} Q(X)
\, |C_X|\, \chi_{\rho}(X) \\
& +  Q(Y_0) \left ((q^2 -1) \; \chi_{\rho} \matt 1101  - (q-1) \; \chi_{\rho} \matt 1001  \right )  \\ 
&   + \sum_{t \in \mathbb F_q^\times} Q(Y_t) \left ((q^2 -1) \; \chi_{\rho}\matt t10t + \chi_{\rho} \matt t00t  \right ),
\end{split}
\end{equation}
and each eigenvalue above appears with multiplicity $(q+1) \dim(\rho).$ 

\item[(c)] For $A = \matt 0000$ and $\rho \in \Sigma_A,$ we have; 
\[
\lambda_\rho = \sum_{X \in \psi } Q(X) |C_X| = 1,
\]
and this occurs with multiplicity one, as expected by Proposition~\ref{prop:statdist-gl2}.
 
\end{itemize} 

Therefore we have a total of 
\[
|\GL_2(\mathbb F_q)| + (q+1) \left (1 + q + (q-2)(q+1)  \right )  + 1  =    q^4  = |M_2(\mathbb F_q)|
\]
eigenvalues of $B_R,$ as expected. Recall from the paragraph immediately following \eqref{mult part of transition matrix} that the eigenvalues of $M_R$ are easily obtained from those of $B_R.$

\subsection{Stationary distribution}
\label{sec:m2 statdist}

\begin{prop}
\label{prop:statdist-gl2}
The stationary distribution of $\Xu t$ has the following formula:
\[
\pi(x) = 
\begin{cases}
\ds \frac{\alpha}{q^3+q^2 - q - (q^2-1)(q^2-q) \alpha}, & x \in U_R, \\[0.5cm]
\ds \frac{q^2 \alpha}
{(1 + (q^2-1) \alpha) (q^3+q^2 - q - (q^2-1)(q^2-q) \alpha)}, & x \notin U_R, x \neq 0, \\[0.5cm]
\ds \frac{q^3+q^2 - q - q(q^2-1) \alpha}{(1 + (q^2-1) \alpha) (q^3+q^2 - q - (q^2-1)(q^2-q) \alpha)}, & x=0.
\end{cases}	
\]
\end{prop}

\begin{proof}
The formula for units can be obtained by directly from Corollary~\ref{cor:statdistunif-units}. 
If $x$ is a nonzero nonunit, the only other ideal it belongs to is the one generated
by 1. Therefore, from Corollary~\ref{cor:statdistunif}, it follows that
\[
\pi(x) = \frac{\alpha + (1-\alpha) \, |U_R| \,\pi(1)}
{|R|-(1-\alpha)q^2 (q^2-1)} 
= \frac{\alpha + (1-\alpha) \, (q^2-1)(q^2-q) \,\pi(1)}{q^4-(1-\alpha)q^2 (q^2-1)},
\]
where we have used the fact that $|U_x| = |U_R/\Stab(x)| = q^2-1$ and $|\ann(x)| = q^2$ for any nonzero nonunit $x$ by similar computations as those done in Table~\ref{tab:stab}. 
A little simplification gives the result.
The formula for the zero matrix is then a consequence of the total probability being 1.
\end{proof}

\section*{Acknowledgements}
The authors thank an anonymous referee for many useful comments.
The authors were partially supported by the UGC Centre for Advanced Studies. AA was also partly supported by Department of Science and Technology grant EMR/2016/006624. 

\appendix 

\section{Irreducible representations of $\GL_2(\mathbb F_q)$} 
\label{sec:appendix} 

In this section, we briefly describe the irreducible representations and the character table of $\GL_2(\mathbb F_q)$ for the sake of completeness. This is well known in the literature and can be found in many texts. Here we shall refer to \cite{shapiro-1983} for all the details. For an abelian group $A,$ we use $\widehat{A}$ to denote the set of its all one-dimensional representations of $A.$ For a group $G,$ we use $\one_G$ to denote the trivial one-dimensional representation of $G.$ 

\subsection{Principal series representations of $\GL_2(\mathbb F_q)$}
Let $U = \left\{ \matt 1w01  \right\}_{w \in \mathbb F_q}$ be the subgroup of $\GL_2(\mathbb F_q)$ consisting of all unipotent upper triangular matrices. An irreducible representation of $\GL_2(\mathbb F_q)$ is called a {\em principal series representation} if it is an irreducible constituent of $\mathrm{Ind}_U^{\GL_2(\mathbb F_q)}(\one_U).$  Now we describe 
all of the principal series representations of $\GL_2(\mathbb F_q).$  

 Let $B = \left\{ \matt uw0v \right\}_{\substack{ u,v \in \mathbb F_q^\times \\ w \in \mathbb F_q }}$ be the subgroup of $\GL_2(\mathbb F_q)$ consisting of all upper triangular matrices. Let $\mu_1, \mu_2 \in \widehat{ \mathbb  F_q^\times}.$ Define the one-dimensional representation $\mu_{1, 2}$ of $B$ by  
\begin{equation}
\mu_{1,2} : B \rightarrow \mathbb C^\times \,\, ; \,\, \mu_{1, 2} \matt uw0v = \mu_1(u)\mu_2(v).  
\end{equation}
Let $\rho_{\mu_1, \mu_2} = \ind_B^G ( \mu_{1, 2})$ be the corresponding induced representation of $\mathbb GL_2(\mathbb F_q).$  Then we have the following result regarding their mutual relation. 

\begin{prop}
For $\mu_1, \mu_2 \in \widehat{\mathbb F_q^{\times}},$ the following hold.
\begin{enumerate}
	\item $\dim_{\mathbb C} (Hom_G(\rho_{\mu_1, \mu_2} , \rho_{\mu'_1, \mu'_2}  ) \neq 0$ if and only if $\{\mu_1, \mu_2 \} = \{\mu'_1, \mu'_2 \}.$ 
	\item For $\mu_1 \neq \mu_2,$ the representations $\rho_{\mu_1, \mu_2}$ are irreducible. 
	\item For $\mu_1 = \mu_2,$ we have $\dim_{\mathbb C} (End_G(\rho_{\mu_1, \mu_1})) = 2.$  
\end{enumerate}
	
\end{prop}
 The last part of the above proposition implies that for every $\mu \in \widehat{\mathbb F_q^\times}$ the representation $\rho_{\mu, \mu} \cong W_1^\mu \oplus W_2^\mu,$ where $W_1^\mu$ and  $W_2^\mu$ are inequivalent representations of $G.$ Next we describe one of the constituents of $\rho_{\mu,\mu}.$ The other one is the complement and is easily obtained.  For every $\mu \in \widehat{\mathbb F_q^\times},$ define the one-dimensional representation $\mathrm{det}_\mu : \GL_2(\mathbb F_q) \rightarrow \mathbb C^\times$ by $\det_\mu(X) = \mu(\det(X)).$ Then it is easy to see that one of the constituents of $\rho_{\mu, \mu}$ is isomorphic to the one-dimensional representation $\mathrm{det}_\mu.$ For notational convenience, we shall denote $\one_{\mathbb F_q^{\times}}$ by $\one_q.$ We note that $\det_{\one_q} = \one_{\GL_2(\mathbb F_q)}.$ Now onwards, we shall use $\rho_\mu$ to denote a complement of $\mathrm{det}_\mu$ in $\rho_{\mu, \mu}.$
 
 Collecting all this information together, we obtain $2(q-1) + \frac{(q-1)(q-2)}{2}$ inequivalent irreducible representations of $\GL_2(\mathbb F_q).$ We know that the total number of inequivalent irreducible representations of a finite group is equal to the number of its conjugacy classes. Therefore, at this point it is good to compare the number of irreducible representations constructed above with the number of conjugacy classes of $\GL_2(\mathbb F_q)$ so that we know exactly how many more irreducible representations exist for the group $\GL_2(\mathbb F_q).$ This we do in the Table~\ref{tab:Conjugacy Classs}. 

\begin{table}[htbp!]
\[
\begin{array}{|c|c|c|} 
\hline
\mbox{Conjugacy class type} & \mbox{Representative}  & \mbox{Number of classes}\\ 
\hline 
\mbox{central semisimple}  &  \matt x00x_{x \in \mathbb F_q^\times} &  q-1\\
\mbox{unitary}  &  \matt x10x_{x \in \mathbb F_q^\times} &  q-1\\
\mbox{non central semisimple}  & \matt x00y_{x,y \in \mathbb F_q^\times, x \neq y}    &  \frac{(q-1)(q-2)}{2}\\
\mbox{anisotropic}  &  \matt {\alpha}00{\bar{\alpha}}_{\alpha \in \mathbb F_{q^2} \setminus \mathbb F_q} &  \frac{q^{2}-q}{2}\\
 \hline 
\end{array}
\]
\caption{Conjugacy class types of $\GL_2(\mathbb F_q)$ along with their representatives and the number of each type.}
\label{tab:Conjugacy Classs}
\end{table}

This table along with the above discussion on construction of principal series irreducible representations of $\GL_2(\mathbb F_q)$ implies that, we still need to construct $(q^2-q)/2$ irreducible representations of $\GL_2(\mathbb F_q).$ This we obtain in the next section.

\subsection{Cuspidal representations of $\GL_2(\mathbb F_q)$} 
An irreducible representation $\rho$ of $G$ is called 
{\em cuspidal} if it is not a principal series representation. These representations are slightly more complicated to define as compared to those of the principal series. Here we outline their construction just enough to explain the entries of the character table of $\GL_2(\mathbb F_q)$ and urge the interested reader to look at \cite{shapiro-1983} for more details. 

Let $\mathbb F_{q^2}$ be a degree two extension field of $\mathbb F_q.$ Then the map $\sigma: x\mapsto x^q$ is a field automorphism of $\mathbb F_{q^2}$ of order two with the property that $\sigma(x) = x$ implies $x \in \mathbb F_q.$ Then the Norm map $N : \mathbb F_{q^2}^\times \rightarrow \mathbb F_q^\times$ defined by $x \mapsto x\sigma(x)$ is easily see to be a  surjective group homomorphism. Next, for every $\nu \in \widehat{\mathbb F_{q^2}},$ we define $\sigma(\nu) \in \widehat{\mathbb F_{q^2}}$ by $\sigma(\nu) (x) = \nu(\sigma(x)).$ A one-dimensional representation $\nu$ of $\mathbb F_{q^2}^\times$ is called {\em non-decomposable} if $\nu \neq \sigma(\nu).$ It turns out that cuspidal representations are parametrized by the set of non-decomposable characters of $\mathbb F_{q^2}^\times.$ More specifically, the following result is true. 

\begin{thm}
\label{thm:A2}
There exists an irreducible representation $\rho_\nu$ of $\GL_2(\mathbb F_q)$ of dimension $(q-1)$ for each non-decomposable character $\nu$ of $\mathbb F_{q^2}^\times.$ Furthermore $\rho_\nu \cong \rho_{\nu'}$ if and only if $\nu' \in \{\nu, \sigma(\nu) \}.$ 
\end{thm}

It is easy to see that the number of non-decomposable characters of $\mathbb F_{q^2}^\times $ is $q^2 -q.$ Further by Theorem~\ref{thm:A2} for non-decomposable characters $\nu$ of $\mathbb F_{q^2}^\times ,$ there exists exactly one non-decomposable character $\nu'$ of $\mathbb F_{q^2}^\times $ such that $ \nu \neq \nu'$ and $\rho_\nu \cong \rho_{\nu'}.$ So we obtain exactly $(q^2-q)/2$ inequivalent irreducible representations of $\GL_2(\mathbb F_q)$ this way. 
This implies that the set of representations $\rho_\nu$ together with principal series irreducible representations form a complete set of irreducible representations of  $\GL_2(\mathbb F_q).$ It is possible to describe all these cuspidal representations precisely. 
We refer the reader to \cite[Section~13]{shapiro-1983} for more details.  

\subsection{Character table of $\GL_2(\mathbb F_q)$} 
Now we are in a position to describe the character table of $\GL_2(\mathbb F_q).$ By definition if $\rho:G \rightarrow \GL(V)$ is a finite dimensional representation of $G,$ the {\em character} of $\rho,$ denoted $\chi_\rho,$ is a complex valued function  defined by $\chi_\rho (g) = \mathrm{Tr}(\rho(g)).$ It is well-known that the character of a representation is an invariant of a representation determined upto equivalence of representation. Further, two irreducible representations of a finite group are equivalent if and only if their characters are equal. Therefore, the character table encodes important information about the irreducible representations of a group. Note that ${\chi}_\rho(s) =  {\chi}_\rho(t),$ whenever $s$ and $t$ are conjugate. So we shall describe character values only on the conjugacy class representatives of $\GL_2(\mathbb F_q).$ The character table is given in Table~\ref{tab:chartable}.

\begin{table}[htbp!]
\begin{tabular}{|c|c|c|c|c|}
\hline  \label{character-table} 
& $\matt x00x_{x \in \mathbb F_q^\times}$ & $\matt x10x_{x \in \mathbb F_q^\times}$ & $\matt x00y_{x \neq y \in \mathbb F_q^\times}$ & $\matt {\alpha}00{\bar{\alpha}}_{\alpha \in \mathbb F_{q^2} \setminus \mathbb F_q}$  \\
\hline 
$\det_\chi$& $\chi(x)^2$  & $\chi(x)^2$ & $\chi(x) \chi(y)$ &  $\chi(\alpha \bar{\alpha})$ \\
\hline 
$\rho_\chi$ &  $q \chi(x)^2$ & 0 & $\chi(xy)$ & $-\chi(\alpha \bar{\alpha})$ \\
\hline 
$\rho_{\chi_1, \chi_2}$ &$(q+1) \chi_1(x) \chi_2(x)$ & $\chi_1(x) \chi_2(x)$ & $\chi_1(x) \chi_2(y) + \chi_1(y) \chi_2(x)$ & 0 \\
\hline 
$\rho_\nu$ & $(q-1) \nu(x)$ & $\nu(x)$ & 0 &$ -\nu(\alpha) - \nu(\bar{\alpha}) $ \\
\hline 
\end{tabular} 
\vspace{0.2cm}
\caption{The character table of $\GL_2(\mathbb F_q).$}
\label{tab:chartable}
\end{table}

The conjugacy classes of $\GL_2(\mathbb F_q)$ are described in Table~\ref{tab:Conjugacy Classs}.
The first three rows of the character table in Table~\ref{tab:chartable} correspond to the principal series irreducible representations and can be obtained by using the character formula for induced representations. The last row corresponds to the cuspidal character and uses the explicit description of cuspidal irreducible representations of $\GL_2(\mathbb F_q)$ as given in \cite[Section~13]{shapiro-1983}.

\bibliography{ringpapers}
\bibliographystyle{siam}

\end{document}